\documentclass[a4paper,11pt]{amsart}

\usepackage{mathtools}
\usepackage{amsmath}
\usepackage{amsfonts}
\usepackage{amssymb}
\usepackage{graphicx}
\usepackage{mathabx}
\usepackage[abbrev]{amsrefs}

\usepackage{color}
\usepackage{soul,xcolor}
\setstcolor{magenta}
\usepackage{microtype}
\usepackage{comment}

\usepackage{amsthm}
\usepackage[all,cmtip]{xy}
\usepackage{tikz-cd}
\usepackage{tikz}
\usetikzlibrary{matrix,arrows,decorations.pathmorphing}

\newcommand{\mbN}{\mathbb{N}}
\newcommand{\mbQ}{\mathbb{Q}}

\newcommand{\mbR}{\mathbb{R}}

\newcommand{\mbZ}{\mathbb{Z}}
\newcommand{\mbP}{\mathbb{P}}

\newcommand{\mcE}{\mathcal{E}}
\newcommand{\mcF}{\mathcal{F}}
\newcommand{\mcH}{\mathcal{H}}
\newcommand{\mcG}{\mathcal{G}}
\newcommand{\mcJ}{\mathcal{J}}

\newcommand{\mcL}{\mathcal{L}}

\newcommand{\mcO}{\mathcal{O}}

\newcommand{\wtX}{\widetilde{X}}
\newcommand{\ovX}{\overline{X}}

\DeclareMathOperator{\Supp}{Supp}
\DeclareMathOperator{\Spec}{Spec}
\DeclareMathOperator{\Sing}{Sing}

\DeclareMathOperator{\Hom}{Hom}

\DeclareMathOperator{\Pic}{Pic}

\DeclareMathOperator{\Tr}{Tr}
\DeclareMathOperator{\vol}{vol}
\newcommand*{\defeq}{\mathrel{\mathop:}=}

\theoremstyle{plain}
\newtheorem{theorem}{Theorem}[section]
\newtheorem{proposition}[theorem]{Proposition}
\newtheorem{lemma}[theorem]{Lemma}
\newtheorem{corollary}[theorem]{Corollary}
\newtheorem{conj}[theorem]{Conjecture}

\theoremstyle{definition}
\newtheorem{definition}[theorem]{Definition}

\theoremstyle{remark}
\newtheorem{remark}[theorem]{Remark}

\newcommand\myurl[1]{\url{#1}}
\DefineSimpleKey{bib}{myurl}
\BibSpec{webpage}{%
  +{}{\PrintAuthors} {author}
  +{,}{ \textit} {title}
  +{}{ \parenthesize} {date}
  +{,}{ \myurl} {myurl}
  +{,}{ } {note}
  +{.}{ } {transition}
}

\title[Effective bounds in positive characteristic]
{Effective bounds on singular surfaces in positive characteristic} 
 
\author{Jakub Witaszek} 
\subjclass[2010]{14E30, 13A35, 14B05, 14F17}
\keywords{positive characteristic, F-pure singularity, surface, Fujita conjecture, Kodaira vanishing}
\address{Department of Mathematics, Imperial College, London, 180 Queen's Gate, 
London SW7 2AZ, UK} 
\email{j.witaszek14@imperial.ac.uk}

\begin{document}
\maketitle
\begin{abstract}
Using the theory of Frobenius singularities, we show that $13mK_X + 45mA$ is very ample for an ample Cartier divisor $A$ on a Kawamata log terminal surface $X$ with Gorenstein index $m$, defined over an algebraically closed field of characteristic $p>5$.
\end{abstract}

\section{Introduction}

The positivity of line bundles is a fundamental topic of research in algebraic geometry. Showing the base point freeness or very ampleness of line bundles allows for the description of the geometry of algebraic varieties. 
 
The motivation for this paper centers around two questions. The first one is the following: given an ample Cartier divisor $A$, find an effective $n \in \mbN$ for which $nA$ is very ample. A famous theorem of Matsusaka states that one can find such $n \in \mbN$ which depends only on the Hilbert polynomial of $A$, when the variety is smooth and the characteristic of the field if equal to zero (\cite{Matsusaka}). This theorem plays a fundamental role in constructing moduli spaces of polarized varieties. In positive characteristic, Koll\'{a}r proved the same statement for normal surfaces (\cite[Theorem 2.1.2]{Kollar85}). 

The second question motivating the results of this paper is the famous Fujita conjecture, which, in characteristic zero, is proved only for curves and surfaces. 
\begin{conj}[Fujita conjecture] Let $X$ be a smooth projective variety of dimension $n$, and let $A$ be an ample Cartier divisor on $X$. Then $K_X + (n+2)A$ is very ample.
\end{conj}
Fujita-type results play a vital role in understanding the geometry of algebraic varieties.

In positive characteristic, the conjecture is known only for curves, and those surfaces which are neither of general type, nor quasi-elliptic. This follows from a result of Shepherd-Barron which says that on such surfaces, rank two vector bundles which do not satisfy Bogomolov inequality are unstable (\cite[Theorem 7]{Shepherd-Barron}). Indeed, the celebrated proof by Reider of the Fujita conjecture for characteristic zero surfaces, can be, in such a case, applied without any modifications (see \cite{Terakawa}, \cite{Reider}).

	Given lack of any progress for positive characteristic surfaces of general type, Di Cerbo and Fanelli undertook a different approach to the problem (\cite{DiCerboFanelli}). They proved among other things that $2K_X + 4A$ is very ample, where $A$ is ample, and $X$ is a smooth surface of general type in characteristic $p\geq 3$. 
	
In this paper, we consider the aforementioned questions for singular surfaces. As far as we know, no effective bounds for singular surfaces in positive characteristic have been obtained before. The main theorem is the following.
 
\begin{theorem} \label{theorem:main-mini} Let $X$ be a projective surface with Kawamata log terminal singularities defined over an algebraically closed field of characteristic $p>5$. Assume that $mK_X$ is Cartier for some $m \in \mbN$. Then, for an ample Cartier divisor $A$:
\begin{itemize} 
	\item $4mK_X+14mA$ is base point free, and
	\item $13mK_X+45mA$ is very ample.
\end{itemize}
\end{theorem}
In particular, Kawamata log terminal surfaces with $K_X$ ample, and defined in characteristic $p>5$, satisfy that $58mK_X$ is very ample, where $m$ is the Gorenstein index. 

The bounds are not sharp. See Theorem \ref{theorem:main} for a slightly more general statement. Instead of assuming that $X$ is Kawamata log terminal and $p>5$, it is enough to assume that $X$ is $F$-pure and $\mbQ$-factorial. The theorem also holds in characteristic zero, but in such a case the existence of effective bounds follows in an easier way by Koll\'{a}r's effective base point free theorem (\cite{kollareffective}) and the Kodaira vanishing theorem.

The proof consists of three main ingredients. First, we apply the result of Di Cerbo and Fanelli on a desingularization of $X$. This shows that the base locus of $2mK_X + 7mA$ is zero dimensional. Then, we apply the technique of Cascini, Tanaka and Xu (see \cite[Theorem 3.8]{CasciniTanakaXu}) to show that the base locus of $2(2mK_X + 7mA)$ is empty. Therewith, the very ampleness of $13mK_X + 45mA$ follows from a generalization of a result of Keeler (\cite[Theorem 1.1]{Keeler}) in the case of $F$-pure varieties.

As far as we know, after the paper of Cascini, Tanaka and Xu had been announced, no one has yet applied their technique of constructing $F$-pure centers. We believe that down-to-earth examples provided in our paper may be suitable as a gentle introduction to some parts of their prolific paper, \cite{CasciniTanakaXu}.\\

As a corollary to the main theorem, we obtain the following Matsusaka-type bounds.
\begin{corollary} \label{corollary:Matsusaka-mini}Let $A$ and $N$ be, respectively, an ample and a nef Cartier divisor on a Kawamata log terminal projective surface defined over an algebraically closed field of characteristic $p>5$. Let $m\in \mbN$ be such that $mK_X$ is Cartier. Then $kA - N$ is very ample for any
\[
k > \frac{2A \cdot (H+N)}{A^2}((K_X +3A)
\cdot A + 1),
\]
where $H \defeq 13mK_X+45mA$.

\end{corollary}

One of the fundamental conjectures in birational geometry is Borisov-Alexeev Boundedness conjecture, which says that $\epsilon$-klt log Fano varieties are bounded. In dimension two it was proved by Alexeev (\cite[Theorem 6.9]{Alexeev}).  One of the ingredients of the proof is the beforementioned result about boundedness of polarized surfaces by Koll\'{a}r (\cite[Theorem 2.1.2]{Kollar85}). Further, explicit bounds on the volume have been obtained by Lai (\cite[Theorem 4.3]{Lai}) and Jiang (\cite[Theorem 1.3]{Jiang}). 

For characteristic $p>5$, in the proof of the boundedness of $\epsilon$-klt log del Pezzo pairs, we can replace Koll\'{a}r's result by Theorem \ref{theorem:main-mini}, and hence obtain rough, but explicit bounds on the size of the bounded family. 

\begin{corollary}\label{corollary:boundedness_of_log_del_Pezzo} 
For any $\epsilon \in \mbR_{>0}$ and a finite set $I \subseteq [0,1] \cap \mbQ$, there exist effectively computable natural numbers $b(\epsilon,I)$ and  $n(\epsilon,I)$ satisfying the following property. 

Let $(X,\Delta)$ be an $\epsilon$-klt log del Pezzo surface with the coefficients of $\Delta$ contained in $I$, defined over an algebraically closed field of characteristic $p>5$. Then, there exists a very ample divisor $H$ on $X$ such that $H^i(X,H) = 0$ for $i>0$ and
\[
|H^2|, |H \cdot K_X|, |H\cdot \Delta|, |K_X \cdot \Delta|, |\Delta^2|
\]
are smaller than $b(\epsilon,I)$. Further, $H$ embeds $X$ into $\mbP^k$, where $k \leq b(\epsilon, I)$, and $n(\epsilon,I)\Delta$ is Cartier.
\end{corollary}

The paper is organized as follows.

The second section consists of basic preliminaries. In the third section we give a proof of the base point free part of the main theorem. In the fourth section, we prove a technical generalization of the main theorem to arbitrary characteristic. In the fifth section we show Matsusaka-type bounds (Corollary \ref{corollary:Matsusaka-mini}). In the sixth section we derive effective bounds on log del Pezzo pairs.\\

As far as we are concerned, the best source of knowledge about Frobenius singularities are unpublished notes of Schwede \cite{SchwedeUtah}. We also recommend \cite{SchwedeTucker}.

The readers interested in Matsusaka theorem and Fujita-type theorems are encouraged to consult \cite[Section 10.2 and 10.4]{Laz2}. A proof of Reider's theorem for normal surfaces in characteristic zero may be found in \cite{Sakai}. Certain other Fujita-type bounds for singular surfaces in characteristic zero are obtained in \cite{LangerAdjoint}. The effective base point free theorem in characteristic zero is proved in \cite{kollareffective}. Various results on the base point free theorem for surfaces in positive characteristic may be found in \cite{tanakaxmethod} and \cite{mnw}.



\section{Preliminaries}
We always work over an algebraically closed field $k$ of positive characteristic $p>0$.

We refer to \cite{KollarMori} for basic results and basic definitions in birational geometry like log discrepancy or Kawamata log terminal singularities. We say that a pair $(X,\Delta)$ is a \emph{log Fano} pair if $-(K_X + \Delta)$ is ample. In the case when $\dim(X)=2$, we say that $(X,\Delta)$ is a \emph{log del Pezzo} pair.

A pair $(X,B)$ is \emph{$\epsilon$-klt}, if the log discrepancy along any divisor is greater than $\epsilon$. Note that the notion of being $0$-klt is equivalent to klt. 

 The \emph{Cartier index} of $(X,\Delta)$ is the minimal number $m \in \mbN$ such that $m(K_X +\Delta)$ is Cartier. If $(X,\Delta)$ is klt, then it must be $1/m$-klt.

Recall that any Kawamata log terminal surface has rational singularities, and is, in particular, $\mbQ$-factorial (see \cite{Tanaka2dim}). 
	
	We denote the base locus of a line bundle $\mcL$ by $\mathrm{Bs}(\mcL)$. Note, that by abuse of notation, we use the notation for line bundles and the notation for divisors interchangeably.

	We will repeatedly use, without mentioning directly, that $K_X + 3A$ is nef, where $X$ is a projective normal surface and $A$ is an ample Cartier divisor. This is a direct consequence of the cone theorem (\cite[Proposition 3.15]{Tanaka2dim}). \\

	The following facts are used in the proofs in this paper.
	
\begin{theorem}[{Mumford regularity, \cite[Theorem 1.8.5]{Laz1}}] \label{theorem:fujita}
 Let $X$ be a projective variety, and let $M$ be a globally generated ample line bundle on $X$. Let $\mcF$ be a coherent sheaf on $X$ such that $H^i(X, \mcF \otimes M^{-i})=0$ for $i>0$. Then $\mcF$ is globally generated.
\end{theorem}

\begin{theorem}[{Fujita vanishing, \cite[Theorem 1.4.35]{Laz1} and \cite[Remark 1.4.36]{Laz1}}] \label{theorem:fujita_vanishing} Let $X$ be a projective variety, and let $H$ be an ample divisor on $X$. Given any coherent sheaf $\mcF$ on $X$, there exists an integer $m(\mcF,H)$ such  that
\[
H^i\big(X, \mcF \otimes \mcO_X(mH + D)\big) = 0,
\]
for all $i>0$, $m \geq m(\mcF,H)$ and any nef Cartier divisor $D$ on $X$.
\end{theorem}

\begin{theorem}[Log-concavity of volume] \label{theorem:volume} For any two big Cartier divisors $D_1$ and $D_2$ on a normal variety $X$ of dimension $n$ we have
\[
\vol(D_1 +D_2)^{\frac{1}{n}} \geq \vol(D_1)^{\frac{1}{n}} + \vol(D_2)^{\frac{1}{n}}.
\]
\end{theorem}
\noindent Recall that 
\[
\vol(D) \defeq \limsup_{m \to \infty} \frac{H^0(X,mD)}{m^n / n!}.
\]
\begin{proof}
See \cite[Theorem 2.2]{DiCerboFanelli} (cf.\ \cite[Theorem 11.4.9]{Laz2} and \cite{TakagiFujita}). 
\end{proof}
\subsection{Frobenius singularities}
All the rings in this section are assumed to be geometric and of positive characteristic, that is finitely generated over an algebraically closed field of characteristic $p>0$.

One of the most amazing discoveries of singularity theory is that properties of the Frobenius map may reflect how singular a variety is. This observation is based on the fact that for a smooth local ring $R$, the $e$-times iterated Frobenius map $R \to F^e_* R$ splits. Further, the splitting does not need to hold when $R$ is singular. 

This leads to a definition of \emph{$F$-split rings}, rings $R$ such that for divisible enough $e \gg 0$ the $e$-times iterated Frobenius map $F^e \colon R \to F^e_* R$ splits. For log pairs, we have the following definition.
\begin{definition} We say that a log pair $(X, \Delta)$ is \emph{$F$-pure}  if for any close point $x \in X$ and any natural number $e > 0 $ there exists a map \[ \phi \in \mathrm{Hom}_{\mcO_{X,x}}(F^e_*\mcO_{X,x}(\lfloor(p^e-1)\Delta \rfloor), \mcO_{X,x}) \] such that $1 \in \phi(F^e_*\mcO_{X,x})$, where $\mcO_{X,x}$ it the stalk at $x \in X$.
\end{definition}

As a consequence of Grothendieck duality (see \cite[Lemma 2.9]{HaconXu}), we have
\[
\Hom(F^e_* \mcO_X, \mcO_X) \simeq H^0(X, \omega_X^{1-p^e}).
\]
This explains the following crucial proposition.
\begin{proposition}[{\cite[Theorem 3.11, 3.13]{SchwedeFAdjunction}}] \label{proposition:bijection_frobenius_divisors} Let $X$ be a normal variety. Then, there is a natural bijection.
\[
\left\{ \begin{array}{c} \text{Non-zero }\mcO_X\text{-linear maps} \\ \phi \colon F^e_* \mcO_X \to \mcO_X \text{ up to} \\ \text{pre-multiplication by units} \end{array} \right\} \longleftrightarrow \left\{ \begin{array}{c} \text{Effective }\mbQ\text{-divisors }\Delta \text{ such that } \\ (1-p^e)\Delta \sim -(1-p^e)K_X \end{array} \right\}
\]
\end{proposition}
The $\mbQ$-divisor corresponding to a splitting $\phi \colon F^e_* \mcO_X \to\mcO_X$ will be denoted by $\Delta_{\phi}$. The morphism extends to $\phi \colon F^e_* \mcO_X((p^e-1)\Delta_{\phi}) \to \mcO_X$, which gives that $(X,\Delta_{\phi})$ is $F$-pure.

Note that we will apply the above proposition mainly for $X$ replaced by $\Spec \mcO_{X, x}$, where $x \in X$.
\\ 

Frobenius singularities are alleged to be the correct counterparts of birational singularities in positive characteristic. This supposition is propped up by the following theorems.
\begin{theorem}[{see \cite{HaraWatanabe}}] \label{theorem:fsplit_and_lc} Let $(X,B)$ be a log pair. If $(X,B)$ is $F$-pure, then $(X, B)$ is log canonical.
\end{theorem}

\begin{theorem}[{\cite{Hara2dimSing}}] \label{theorem:2dimp>5} Let $X$ be a Kawamata log terminal projective surface defined over an algebraically closed field of characteristic $p>5$. Then $X$ is $F$-pure.
\end{theorem} 
The above theorem implies, even more, that $X$ is strongly $F$-regular. In this paper, however, we do not need the notion of $F$-regularity.\\

A key tool in the theory of Frobenius splittings is a trace map (see \cite{TanakaTrace} and \cite{SchwedeAdjoint}). For an integral divisor $D$ on a normal variety $X$, there is an isomorphism derived from the Grothendieck duality (see \cite[Lemma 2.9]{HaconXu}):
\begin{equation} \label{equation:Grothendieck}
\mcH om_{\mcO_X}(F^e_*\mcO_X(D), \mcO_X) \simeq \mcO_X(-(p^e-1)K_X-D).
\end{equation}

\begin{definition} Let $B$ be a $\mbQ$-divisor such that $(p^e-1)(K_X+B)$ is Cartier. We call
\[
\Tr^e_{X,B} \colon F^e_*\mcO_X(-(p^e-1)(K_X + B)) \rightarrow \mcO_X,
\]
the \emph{trace map}. It is constructed by applying the above isomorphism (\ref{equation:Grothendieck}) to the map 
\begin{equation} \label{equation:trace_evaluation}
\mcH om_{\mcO_X}(F^e_*\mcO_X((p^e-1)B), \mcO_X) \xrightarrow{\mathrm{ev}} \mcH om_{\mcO_X}(\mcO_X, \mcO_X)  
\end{equation}
being the dual of the composition $\mcO_X \xrightarrow{F^e} F^e_* \mcO_X \xhookrightarrow{\hphantom{F^e}} F^e_* \mcO_X((p^e-1)B)$.
\end{definition}

The rank one sheaves in question are not necessary line bundles, but since $X$ is normal, we can always restrict ourselves to the smooth locus.  The trace map can be also defined when the index of $K_X+B$ is divisible by $p$, but we will not need it in this paper.

 The following proposition reveals the significance of the trace map.
\begin{proposition}[{\cite[Proposition 2.10]{HaconXu}}] \label{proposition:trace_and_fsplit} Let $(X,B)$ be a normal log pair such that the Cartier index of $K_X+B$ is not divisible by $p$. Then $(X,B)$ is $F$-pure at a point $x \in X$ if and only if the trace map $\Tr^e_{X,B}$ is surjective at $x$ for all enough divisible $e \gg 0$.
\end{proposition}
It is easy to see that $\Tr^e_{X,B}$ is surjective at $x$ for all divisible enough $e \gg 0$ if and only if it is surjective for just one $e>0$ satisfying that $(p^e-1)B$ is Weil. For the convenience of the reader we give a proof of the proposition.
\begin{proof}
The key point is that $\Tr^e_{X,B}$ is induced by the evaluation map (\ref{equation:trace_evaluation}). Replace $X$ by $\Spec \mcO_{X,x}$. For $\phi \in \mathrm{Hom}_{\mcO_X}(F^e_*\mcO_X((p^e-1)B), \mcO_X)$, the image $\mathrm{ev}(\phi)$ is defined by the commutativity of the following diagram:
\begin{center}
\begin{tikzcd}
\mcO_X \arrow{r}{F^e} \arrow[bend right = 25]{rr}{\mathrm{ev}(\phi)} & F^e_* \mcO_X((p^e-1)B) \arrow{r}{\phi} & \mcO_X.
\end{tikzcd}
\end{center}
Note that $\mathrm{Hom}(\mcO_X,\mcO_X) \simeq \mcO_X$ is generated by the identity morphism $\mathrm{id}$. In particular, $\mathrm{ev}$ is surjective if and only if there exists $\phi$ such that $\mathrm{ev}(\phi)=\mathrm{id}$, which is equivalent to $\phi$ being a splitting.
\end{proof}

Further, we consider another version of the trace map. Let $D$ be a $\mbQ$-divisor such that $K_X + B + D$ is Cartier. Tensoring the trace map $\Tr^e_{X,B}$ by it, we obtain:
\[
\Tr^e_{X,B}(D) \colon F^e_*\mcO_X(K_X + B + p^eD) \longrightarrow \mcO_X(K_X + B + D).
\]
By abuse of notation, both versions of the trace map are denoted in the same way.

Later, we will need the following lemma.

\begin{lemma} \label{lemma:index_divisible_by_p} Let $X$ be a $\mbQ$-factorial variety, which is $F$-pure at a point $x \in X$. Then, there exists an effective $\mbQ$-divisor $B$ such that \[ (p^e-1)(K_X + B) \] is Cartier for enough divisible $e \gg 0$, and $(X,B)$ is $F$-pure at $x$. One can take the coefficients of $B$ to be as small as possible.
\end{lemma}
More precisely, if we can find $B$ as above, then there exists a sequence $\lim_{j \to \infty} a_j = 0$ such that $a_j B$ satisfies the conditions of this lemma.

Note that, if the Gorenstein index of $X$ is not divisible by $p$, then one can just take $B=0$.
\begin{proof}
By Proposition \ref{proposition:bijection_frobenius_divisors}, we know that there exists an effective $\mbQ$-divisor $\Delta$ such that $(X,\Delta)$ is $F$-pure at $x \in X$, and the Cartier index of $K_X + \Delta$ at $x$ is not divisible by $p$. 

In particular, for enough divisible $e \gg 0$, there exists a $\mbQ$-divisor $D \sim (p^e-1)(K_X+\Delta)$ such that $x \not \in \Supp D$. Further, we can find a Cartier divisor $E$ disjoint with $x$ for which $E + D$ is effective.

Notice that $K_X + \Delta + D + E \sim p^e(K_X+\Delta) + E$ has Cartier index indivisible by $p$ for $e \gg 0$. Take
\[
B = \frac{1}{p^n+1}(\Delta + D + E)
\] 
for $n \gg 0$. Then
\[
K_X + B \sim K_X + \Delta + D + E - \frac{p^n}{p^n+1}(\Delta + D + E),
\]   
where both $K_X + \Delta + D + E$ and $p^n(\Delta + D + E)$ have Cartier indices not divisible by $p$.
\end{proof}

\subsection{Reider's analysis}
Reider's analysis is a method of showing that divisors of the form $K_X + L$ are globally generated or very ample, where $L$ is a big and nef divisor on a smooth surface $X$. 

The idea is that a base point of $K_X + L$ provides us with a rank two vector bundle $\mcE$ which does not satisfy Bogomolov inequality $c_1(\mcE)^2 \leq 4c_2(\mcE)$. In characteristic zero such vector bundles are unstable. Using the instability, one can deduce a contradiction to the existence of a base point, when $L$ is ``numerically-ample enough''.

In positive characteristic, the aforementioned fact about unstable vector bundles on smooth surfaces is not true in general. However, Shepherd-Barron proved it for surfaces which are neither of general type nor quasi-elliptic of Kodaira dimension one (\cite{Shepherd-Barron}). This leads to the following.   

\begin{proposition}[{\cite[Theorem 2.4]{Terakawa}}] \label{proposition:terakawa} Let $X$ be a smooth projective surface neither of general type nor quasi-elliptic with $\kappa(X)=1$, and let $D$ be a nef divisor such that $D^2 > 4$. Assume that $q \in X$ is a base-point of $K_X + D$. Then, there exists an integral curve $C$ containing $q$, such that $D \cdot C \leq 1$.
\end{proposition}
In particular, for such surfaces, $K_X + 3A$ is base point free for an ample divisor $A$.

The goal of this subsection is to prove the following proposition, which covers all types of surfaces. 
\begin{proposition} \label{proposition:andrea_general} Let $X$ be a smooth projective surface, and let $D$ be a nef and big divisor on it. Assume that $q \in S$ is a base point of $K_X + D$. Further, suppose that
\begin{enumerate}
	\item $D^2 > 4$, if $X$ is quasi-elliptic with $\kappa(X)=1$,
	\item $D^2 > \vol(K_X) + 4$, if $X$ is of general type and $p \geq 3$, 
	\item $D^2 > \vol(K_X) + 6$, if $X$ is of general type and $p=2$, or
	\item $D^2 > 4$, otherwise.	
\end{enumerate}
Then, there exists a curve $C$ containing $q$ such that
\begin{enumerate}
	\item[(1a)] $D \cdot C \leq 5$, if $X$ is quasi-elliptic with $\kappa(X)=1$ and $p=3$,
	\item[(1b)] $D \cdot C \leq 7$, if $X$ is quasi-elliptic with $\kappa(X)=1$ and $p=2$,
	\item[(2)] $D \cdot C \leq 1$, if $X$ is of general type and $p\geq 3$,	
	\item[(3)] $D \cdot C \leq 7$, if $X$ is of general type and $p=2$, or
	\item[(4)] $D \cdot C \leq 1$, otherwise.	
\end{enumerate}
\end{proposition}
Note that the case (4) is nothing else but Proposition \ref{proposition:terakawa}. The proof follows step-by-step the proof by Di Cerbo and Fanelli (\cite{DiCerboFanelli}). The only addition is that the curve $C$ must contain $q$. The idea that this must hold has been established in (\cite{Sakai}) based on (\cite{Serrano}), but, for the convenience of the reader, we present the full proof below.

The following is crucial in the proof of Proposition \ref{proposition:andrea_general}.
\begin{proposition}[{\cite{DiCerboFanelli}}] \label{proposition:key_for_andrea} Consider a birational morphism $\pi \colon Y \to X$ between smooth projective surfaces $X$ and $Y$ which are either of general type, or quasi-elliptic with $\kappa(X)=1$. Let $\overline{D}$ be a big divisor on $Y$ such that $H^1(Y, \mcO_Y(-\overline{D})) \neq 0$ and $\overline{D}^2>0$. Further, suppose that $D \defeq \pi_* \overline{D}$ is nef, and 
\begin{enumerate}
	\item $\overline{D}^2 > \vol(K_X)$, if $X$ is of general type and $p\geq 3$, or
	\item $\overline{D}^2 > \vol(K_X) + 2$, if $X$ is of general type and $p=2$. 
\end{enumerate}

	Then, there exists a non-zero non-exceptional effective divisor $\overline{E}$ on $Y$, such that
	\begin{align*}
		&k\overline{D} - 2\overline{E} \text{ is big,} \\		
		&(k\overline{D} - \overline{E}) \cdot \overline{E} \leq 0, \text{ and } \\
		&0 \leq D \cdot E \leq \frac{k\alpha}{2}-1,
	\end{align*}
where $E \defeq \pi_* \overline{E}$, $\alpha \defeq D^2 - \overline{D}^2$, and 
\begin{itemize}
	\item $k=3$, if $X$ is quasi-elliptic with $\kappa(X)=1$ and $p=3$,
	\item $k=4$, if $X$ is quasi-elliptic with $\kappa(X)=1$ and $p=2$,
	\item $k=1$, if $X$ is of general type and $p\geq 3$, or
	\item $k=1$ or $k=4$, if $X$ is of general type and $p=2$.
\end{itemize}
\end{proposition}
\begin{proof} It follows directly from \cite[Proposition 4.3]{DiCerboFanelli}, \cite[Theorem 4.4]{DiCerboFanelli}, \cite[Proposition 4.6]{DiCerboFanelli} and \cite[Corollary 4.8]{DiCerboFanelli}.
\end{proof}

Further, we need the following lemma.
\begin{lemma}[{\cite[Lemma 2]{Sakai}}]  \label{lemma:sakai} Let $D$ be a nef and big divisor on a smooth surface $S$. If 
\[
D \equiv D_1 + D_2
\]
for numerically non-trivial pseudo-effective divisors $D_1$ and $D_2$, then $D_1 \cdot D_2 > 0$.
\end{lemma}

Now, we can proceed with the proof of the main proposition in this subsection.

\begin{proof}[Proof of Proposition \ref{proposition:andrea_general}]
The first case is covered by Proposition \ref{proposition:terakawa}, so we may assume that $X$ is of general type or quasi-elliptic with $\kappa(X)=1$.

Let $\pi \colon Y \to X$ be a blow-up at $q \in X$ with the exceptional curve $F$. Given that $q$ is a base point of $K_X + D$, from the exact sequence
\[
0 \rightarrow \mcO_Y(\underbrace{\pi^*(K_X + D) -F}_{K_Y + \pi^*D - 2F}) \longrightarrow \mcO_Y(\pi^*(K_X + D)) \longrightarrow \mcO_F(\pi^*(K_X+D)) \rightarrow 0
\]
we obtain that
\[
H^1(Y, \mcO_Y(-\pi^*D + 2F)) = H^1(Y, \mcO_Y(K_Y + \pi^*D -2F)) \neq 0. 
\]
Set $\overline{D} \defeq \pi^*D - 2F$. Since
\[
\overline{D}^2 = D^2 - 4,
\]
we have $\overline{D}^2>0$, and the assertions $(1)$ and $(2)$ in Proposition \ref{proposition:key_for_andrea} are satisfied. Further, using that $H^0(Y, \overline{D}) = H^0(X, \mcO_X(D)\otimes m_q^2)$ and $\vol(D)>4$, one can easily check that $\overline{D}$ is big.  Hence, by Proposition \ref{proposition:key_for_andrea}, there exists a non-zero non-exceptional effective divisor $\overline{E}$ on $Y$, such that
\begin{align*}
&k\overline{D} - 2\overline{E} \text{ is big, and} \\
&(k\overline{D} - \overline{E}) \cdot \overline{E} \leq 0, \text{ and} \\
&0 \leq D \cdot E \leq 2k-1, \\
\end{align*}
where $E = \pi_* \overline{E}$.

To finish the proof, it is enough to show that $\overline{E}$ contains a component, which intersects $F$ properly. Its pushforward onto $X$ would be the sought-for curve $C$. 

Assume that the claim is not true, that is $\overline{E} = \mu^*E + aF$ for $a\geq 0$. We have that
\[
0 \geq (k\overline{D} - \overline{E})\cdot \overline{E} = (kD-E) \cdot E + (2k+a)a. 
\]

This implies $kD\cdot E \leq E^2$. Since $D\cdot E \geq 0$, it holds that $E^2 \geq 0$. Given $kD - 2E$ is big, we may apply Lemma \ref{lemma:sakai} with $kD = (kD-2E) + 2E$, and obtain $kD \cdot E > 2E^2$. This is a contradiction with the other inequalities in this paragraph.
\end{proof}

\section{Base point freeness}
The goal of this section is to prove the base point free part of Theorem \ref{theorem:main-mini}.

\begin{lemma} \label{lemma:cone_corollary} Let $L$ be an ample Cartier divisor on a normal projective surface $X$. Let $\pi \colon \wtX \to X$ be the minimal resolution of singularities. Then $K_{\wtX} + 3\pi^*L$ is nef, and $K_{\wtX} + n\pi^*L$ is nef and big for $n\geq 4$.
\end{lemma}
\begin{proof}

Take a curve $C$. We need to show that $(K_{\wtX} + 3\pi^*L) \cdot C \geq 0$. If $K_{\wtX} \cdot C \geq 0$, then the inequality clearly holds. Thus, by cone theorem (\cite[Theorem 3.13]{Tanaka2dim} and \cite[Remark 3.14]{Tanaka2dim}), we need to prove it, when $C$ is an extremal ray satisfying $K_{\wtX} \cdot C < 0$. In such a case, we have that $K_{\wtX} \cdot C \geq -3$.

If $C$ is not an exceptional curve, then $3\pi^*L \cdot C \geq 3$, and so the inequality holds. But $C$ cannot be exceptional, because then its contraction would give a smooth surface (see \cite[Theorem 1.28]{KollarMori}), and so $\wtX$ would not be a minimal resolution. This concludes the first part of the lemma. 

As for the second part, $K_{\wtX} + n\pi^*L$ is big and nef for $n \geq 4$, since adding a nef divisor to a big and nef divisor gives a big and nef divisor.
\end{proof}

The following proposition yields the first step in the proof of Theorem \ref{theorem:main-mini}.

\begin{proposition} \label{proposition:part_one} Let $X$ be a normal projective surface defined over an algebraically closed field of characteristic $p>3$. Assume that $mK_X$ is Cartier for some $m \in \mbN$. Let $A$ be an ample Cartier divisor on $X$. Then 
\[
\mathrm{Bs}(m(aK_{X} + bA) + N) \subseteq \Sing(X),
\]
for any nef Cartier divisor $N$, where $a=2$ and $b=7$.
\end{proposition}
The proposition is even true for $a=2$ and $b=6$, but in this case, $aK_X + bA$ need not be ample.
\begin{proof}
Let $\pi \colon \overline{X} \to X$ be the minimal resolution of singularities with the exceptional locus $E$. First, we claim that
\[
\mathrm{Bs}(2K_{\ovX} + 7\pi^*A + M) \subseteq E.
\] 
for any nef Cartier divisor $M$ on $\ovX$. Assume this is true. If $m=1$, then the proposition follows automatically. In general, we note that $2K_{\ovX} + 7\pi^*A$ is nef by Lemma \ref{lemma:cone_corollary}, and so setting $M = (m-1)(2K_{\ovX} + 7\pi^*A)+\pi^*N$ yields 
\[
\mathrm{Bs}(m(2K_{\ovX} + 7\pi^*A)+\pi^*N) \subseteq E.
\]
In particular, $\mathrm{Bs}(m(2K_{X} + 7A)+N) \subseteq \pi(E)$, which concludes the proof. 

Hence, we are left to show the claim. Assume by contradiction that there exists a base point $q \in \ovX$ of $2K_{\ovX} + 7\pi^*A + M$ such that $q \not \in E$.

We apply Proposition \ref{proposition:andrea_general} for $D = K_{\ovX} + 7\pi^*A +M$. The assumptions are satisfied, because, by Lemma \ref{lemma:cone_corollary}, $D$ is big and nef, and, by Theorem \ref{theorem:volume},
\begin{align*}
\vol(D) &\geq \vol(K_{\ovX}) + 49, \text{ if }X \text{ is of general type, and }\\
\vol(D) &\geq \vol(K_{\ovX}+4\pi^*A) + \vol(3\pi^*A) > 9 \text{ in general}.
\end{align*}
Here, we used that $K_{\ovX} + 4\pi^*A$ is nef and big by Lemma \ref{lemma:cone_corollary}.  

Therefore, there exists a curve $C$ containing $q$ such that
\[
C \cdot  D \leq 1.
\]
We can write $D = (K_{\ovX} + 3\pi^*A+M) + 4\pi^*A$. As $C$ is not exceptional, $C \cdot \pi^*A > 0$. Thus, we obtain a contradiction.
\end{proof}

Applying above Proposition \ref{proposition:part_one} and Theorem \ref{theorem:2dimp>5}, the base point free part of  Theorem \ref{theorem:main-mini} follows from the following proposition by taking $L \defeq m(aK_X + bA) - K_X$.

\begin{proposition} \label{proposition:part_two} Let $X$ be an $F$-pure $\mbQ$-factorial projective surface defined over an algebraically closed field of characteristic $p>0$. Let $L$ be an ample $\mbQ$-divisor on $X$ such that $K_X + L$ is an ample Cartier divisor and
\[
\mathrm{Bs}(K_X + L + M) \subseteq \Sing(X),
\]
for every nef Cartier divisor $M$. Then $2(K_X + L)+N$ is base point free for every nef Cartier divisor $N$.
\end{proposition}
If we assume that $\dim \mathrm{Bs}(K_X+L+M) = 0$, then the same proof will give us that $3(K_X+L)+N$ is base point free.

Before proceeding with the proof, we would like to give an example explaining an idea of how to show the proposition if we worked in characteristic zero.
\begin{remark} \label{example:char0} Here, $X$ is a smooth surface defined over an algebraically closed field $k$ of characteristic zero, and $L$ is an ample Cartier divisor on it. The goal of this remark is to prove the following statement  by applying a well-known strategy:
\begin{quote}\textit{if $K_X + L$ is ample and $\dim \mathrm{Bs}(K_X + L) = 0$, then $3(K_X + L)$ is base point free.}
\end{quote}

Take any point $q \in \mathrm{Bs}(K_X + L)$. It is enough to show that $3(K_X+L)$ is base point free at $q$. By assumptions, $K_X+L$ defines a finite map outside of its zero dimensional base locus, and so there exist divisors $D_1, D_2, D_3 \in |K_X+L|$ without common components such that the multiplier ideal sheaf $\mcJ(X,\Delta)$ for $\Delta = \frac{2}{3}(D_1 + D_2+ D_3)$ satisfies
\begin{align*}
\dim\, &\mcJ(X,\Delta) = 0, \text{ and }\\
 q \in \ &\mcJ(X,\Delta) .
\end{align*}
Note that $\Delta \sim_{\mbQ} 2(K_X +L)$.

Let $W$ be a zero-dimensional subscheme defined by $\mcJ(X,\Delta)$. We have the following exact sequence
\[
0 \to \mcO_X(\hspace{0.5pt}3(K_X+L)) \otimes \mcJ(X,\Delta) \to \mcO_X(\hspace{0.5pt}3(K_X+L)) \to \mcO_W(\hspace{0.5pt}3(K_X + L)) \to 0.
\]

Since $3(K_X+L) \sim_{\mbQ} K_X + \Delta + L$, by Nadel vanishing theorem (\cite[Theorem 9.4.17]{Laz2})
\[
H^1(X, \mcO_X(\hspace{0.5pt}3(K_X + L)) \otimes \mcJ(X,\Delta)) = 0,
\]
and so
\[
H^0(X, \mcO_X(\hspace{0.5pt}3(K_X+ L))) \longrightarrow H^0(W, \mcO_W(\hspace{0.5pt}3(K_X + L)))
\]
is surjective. Since $\dim W = 0$, we get that $3(K_X + L)$ is base point free along $W$, and so it is base point free at $q$.
\end{remark}

\begin{proof}[Proof of Proposition \ref{proposition:part_two}]
Take an arbitrary closed point $q \in X$. We need to show that $q \not \in \mathrm{Bs}(2(K_X +L)+N)$. By taking $M = K_X + L + N$ in the assumption, we get that $\mathrm{Bs}(2(K_X + L) + N) \subseteq \Sing(X)$. Hence, we can assume $q \in \Sing(X)$.

By assumptions, $K_X+L$ defines a finite map outside of its zero dimensional base locus, so there exist divisors $D_1, D_2 \in |K_X+L|$ such that $\dim (D_1 \cap D_2) = 0$ and $q \in D_1\cap D_2$. Let $W$ be the scheme defined by the interesection of $D_1$ and $D_2$. By definition, $I_W = I_{D_1} + I_{D_2}$.

By Theorem \ref{theorem:fujita_vanishing}, we can choose $e>0$ such that 
\[
H^1\Big(X, \mcO_X\Big(\frac{p^e-1}{2}L +M\Big) \otimes I_W\Big) = 0
\]
for any nef Cartier divisor $M$.

By Lemma \ref{lemma:index_divisible_by_p}, we know that there exists an effective $\mbQ$-divisor $B$ such that
\[
(p^e-1)(K_X + B) 
\]
is Cartier, and
\[
	\mathrm{Tr}_{X,B} \colon F^e_*\mcO_X(-(p^e-1)(K_X+B)) \longrightarrow \mcO_X
\]
is surjective at $q$, for enough divisible $e \gg 0$. If the Gorenstein index of $X$ is not divisible by $p$, then we can take $B=0$. Further, we may assume that $\frac{1}{2}L - B$ is ample. 

Now, take maximal $\lambda_1, \lambda_2 \in \mbZ_{\geq 0}$ such that
\[
\mathrm{Tr}_{X,\Delta} \colon F^e_* \mcL \to \mcO_X
\]
is surjective at the stalk $\mcO_{X,q}$, where 
\begin{align*}
\mcL &\defeq \mcO_X(-(p^e-1)(K_X + B) - \lambda_1D_1 - \lambda_2 D_2)\text{, and} \\
\Delta &\defeq B+\frac{\lambda_1}{p^e-1}D_1 + \frac{\lambda_2}{p^e-1} D_2.
\end{align*}

The pair $(X,\Delta)$ is $F$-pure by Proposition {\ref{proposition:trace_and_fsplit}}. We want to show the existence of the following diagram:
\begin{center}\begin{tikzcd}
F^e_* \mcL \arrow{r} \arrow{d}{\Tr_{X,\Delta}} & F^e_* \big(\mcL|_{W}\big) \arrow{d} \\
\mathcal{O}_X \arrow{r} & \mathcal{O}_{X,q}/m_q 
\end{tikzcd}\end{center}
To show that such a diagram exists we need to prove that the image of $F^e_* (\mcL \otimes I_W)$ under $\mathrm{Tr}_{X,\Delta}$ is contained in $m_q$. This follows from the fact that $I_W  = \mcO(-D_1) + \mcO(-D_2)$ and from the maximality of $\lambda_1, \lambda_2$. More precisely the image of 
\[
F^e_* \mathcal{O}_X(-(p^e-1)(K_X + B) - (\lambda_1+1)D_1 - \lambda_2 D_2)
\]
must be contained in $m_q$, and analogously for $\lambda_2$ replaced by $\lambda_2+1$.

So, we tensor this diagram by the line bundle $\mcO_X(K_X + \Delta + H)$, where
\[
	H \defeq  2(K_X + L) - (K_X + \Delta) + N,
\]
and take $H^0$ to obtain the diagram
\begin{center}\begin{tikzcd}
H^0\big(X, F^e_*\mcO_X(K_X + \Delta + p^eH)\big) \arrow{d}{} \arrow{r} & H^0\big(W, \mathcal{O}_W\big) \arrow{d} \\
H^0\big(X,\mcO_X(K_X + \Delta + H)\big) \arrow{r} & H^0\big(q,\mathcal{O}_{X,q}/m_q \big),
\end{tikzcd}
\end{center}

Note that $K_X + \Delta + H = 2(K_X + L) + N$. Further, by Theorem \ref{theorem:fsplit_and_lc}, $(X, \Delta)$ is log canonical at $q \in X$, and so by Lemma \ref{lemma:log_canonical_coefficients} we get
\[
\frac{\lambda_1}{p^e-1} + \frac{\lambda_2}{p^e-1} \leq 1.
\]  
Therefore, $H$ is ample and \[K_X + \Delta +p^eH \sim \frac{p^e{-}1}{2}L + \underbrace{(p^e{-}1)\Big(\frac{1}{2}L - B\Big) {+} (p^e{+}1 {-} \lambda_1 {-} \lambda_2)(K_X + L) {+} p^eN}_{\mathrm{nef}}.\]

The right vertical arrow is surjective, since $\mathrm{Tr}_{X,\Delta} \colon F^e_* \mcL \to \mcO_X$ is surjective, and $\dim W = 0$. The upper horizontal arrow is surjective as \[H^1\Big(X, \mcO_X\Big(\frac{p^e-1}{2}L + M\Big) \otimes I_W\Big)=0\]
for any nef Cartier divisor $M$. Thus, the lower horizontal arrow is surjective, and so the proof of the base point freeness is completed.
\end{proof}

The following lemma was used in the proof.
\begin{lemma} \label{lemma:log_canonical_coefficients} Let $(X,B + a_1D_1 +a_2D_2)$ be a log canonical two dimensional pair, such that $B$ is an effective $\mbQ$-divisor, $a_1,a_2 \in \mbR_{\geq 0}$, and $D_1$ together with $D_2$ are Cartier divisors intersecting at a singular point $x \in X$. Then $a_1 + a_2 \leq 1$. 
\end{lemma}
Of course, the lemma is not true, when $x$ is a smooth point.
\begin{proof}
Consider a minimal resolution of singularities $\pi \colon \wtX \to X$. Write
\[
K_{\wtX} + \Delta_{\wtX} + \pi^*(B + a_1D_1 + a_2D_2) = \pi^*(K_X + B + a_1D_1 + a_2D_2). 
\]
Since $\pi$ is a minimal resolution, we have that $\Delta_{\wtX} \geq 0$. 

Take an exceptional curve $C$ over $x$. Since $D_1$ and $D_2$ are Cartier, the coefficient of $C$ in $\Delta_{\wtX} + \pi^*(B + a_1D_1 + a_2D_2)$ is greater or equal $a_1 + a_2$. Since $(\wtX,\Delta_{\wtX} + \pi^*(B + a_1D_1 + a_2D_2))$ is log canonical, this concludes the proof of the lemma.
\end{proof}

\subsection{Very ampleness}
The goal of this subsection is to show the following proposition and finish off the proof of Theorem \ref{theorem:main-mini}. 
\begin{proposition}[{cf.\ \cite[Corollary 4.5]{SchwedeAdjoint}, \cite[Theorem 1.1]{Keeler}}] \label{proposition:very_ampleness}
 Let $X$ be an $F$-pure projective variety of dimension $n$. Let $D$ be an ample $\mbQ$-Cartier divisor such that $K_X+D$ is Cartier, and let  $L$ be an ample globally generated Cartier divisor. Then $K_X + (n+1)L + D$ is very ample.
\end{proposition}
The proof follows very closely the strategy described in \cite{Keeler}. Theorems of a similar flavour have been obtained by Schwede (\cite{SchwedeAdjoint}), Smith (\cite{Smith1} and \cite{Smith2}) and Hara (\cite{Hara}). 

First, we need a slight generalization of \cite[Example 1.8.22]{Laz1}.
\begin{lemma}[{c.f.\ \cite[Examples 1.8.18 and 1.8.22]{Laz1}}] \label{lemma:very_ampleness} Let $X$ be a normal projective variety of dimension $n$. Consider a coherent sheaf $\mcF$ and a point $x \in X$. Let $B$ be a globally generated ample line bundle. If
\[
H^{i+k-1}\big(X, \mcF \otimes B^{-(i+k)}\big) = 0,
\] 
for $1 \leq i \leq n$ and $1 \leq k \leq n$, then $\mcF \otimes m_x$ is globally generated.
\end{lemma}
\begin{proof}
Set $\mcF(-i) \defeq F \otimes B^{-i}$. Our goal is to prove that \[H^i(X, \mcF(-i) \otimes m_x) = 0\] for all $i>0$. Then, Theorem \ref{theorem:fujita} would imply the global generatedness of $F \otimes m_x$. 

Since $B$ is ample and globally generated, it defines a finite map and so there exist sections $s_1, s_2, \ldots, s_n \in H^0(X,B)$ intersecting in a zero dimensional scheme $W$ containing $x$. By the same argument as in \cite[Example 1.8.22]{Laz1}, using \cite[Proposition B.1.2(ii)]{Laz1} we get that
\[
H^i(X, \mcF(-i) \otimes I_W) = 0, \text{ for } i>0.
\]

To conclude the proof of the lemma, we consider the following short exact sequence
\[
0 \longrightarrow I_W \longrightarrow m_x \longrightarrow m_x/ I_W \longrightarrow 0,
\]
and tensor it by $\mcF(-i)$, to get a short exact sequence 
\[
0 \longrightarrow \mcG \longrightarrow  \mcF(-i) \otimes I_W  \longrightarrow  \mcF(-i)  \otimes m_x \longrightarrow \mcF(-i) \otimes \big(m_x/ I_W \big)   \longrightarrow 0,
\]
where the term 
\[
\mcG \defeq \ker \big(\mcF(-i) \otimes   I_W  \longrightarrow \mcF(-i) \otimes m_x  \big)
\]
comes from the fact that $\mcF$ may not be flat. Since $m_x/ I_W$ is flat off $W$, we have that 
\[
\dim \Supp \big( \mathrm{Tor}^1 \big(\mcF (-i), m_x/ I_W\big)\big) = 0,\]
and so $H^i(X, \mcG) = 0$ for $i>0$. A simple diagram chasing shows that $H^i(X, \mcF(-i) \otimes m_x  )=0$ for $i>0$, and so we are done.

\end{proof}

\begin{proof}[Proof of Proposition \ref{proposition:very_ampleness}]
Choose a point $q \in X$. To prove the theorem, it is enough to show that $\mcO_X\big(K_X + (n+1)L + D\big) \otimes m_x$ is globally generated at $q$ for all $x \in X$. Set $\mcL \defeq \mcO_X(L)$.

Since $X$ is F-pure, Proposition \ref{proposition:trace_and_fsplit} and Lemma \ref{lemma:index_divisible_by_p} imply that there exists an effective $\mbQ$-divisor $B$ such that 
\[
(p^e-1)(K_X + B)
\]
is Cartier for divisible enough $e>0$ and $\mathrm{Tr}^e_{X,B}$ is surjective at $q$. If the Gorenstein index of $X$ is not divisible by $p$, then we can take $B=0$. Further, we may assume that $D-B$ is ample.
 
Tensoring $\mathrm{Tr}^e_{X,B}((n+1)L + D-B)$ by $m_x$, we obtain a morphism
\[
F^e_* \mcO_X\big(K_X + B + p^e\big((n+1)L + D-B\big)\big) \otimes m_x \longrightarrow \mcO_X\big(K_X + (n+1)L + D\big) \otimes m_x,
\]
which is surjective at $q$, and so it is enough to show that $F^e_* \mcO_X\big(K_X + B + p^e\big((n+1)L + D-B\big)\big) \otimes m_x$ is globally generated for divisible enough $e>0$. 

However, this follows from Lemma \ref{lemma:very_ampleness} as
\begin{align*}
H^{i+k-1}\big(X&, F^e_* \mcO_X\big(K_X + B + p^e\big((n+1)L + D-B\big)\big) \otimes \mcL^{-(i+k)}\big) \\
&= H^{i+k-1}\big(X, \mcO_X\big(K_X + B + p^e\big((n+1 - i-k)L + D-B\big)\big)\big) = 0
\end{align*}
for $e \gg 0$ and $1 \leq i+k -1 \leq n$, by Serre vanishing.
\end{proof}

Now, the proof of the main theorem is straightforward.
\begin{proof}[Proof of Theorem \ref{theorem:main-mini}]
It follows directly from Theorem \ref{theorem:2dimp>5}, Proposition \ref{proposition:part_one}, Proposition \ref{proposition:part_two} and Proposition \ref{proposition:very_ampleness}.
\end{proof}

\section{Generalizations of the main theorem}
In this section we present a technical generalisation of Theorem \ref{theorem:main-mini}.

\begin{theorem} \label{theorem:main} Let $X$ be an $F$-pure $\mbQ$-factorial projective surface defined over an algebraically closed field of characteristic $p>0$. Assume that $mK_X$ is Cartier for some $m \in \mbN$. Let $L$ be an ample Cartier divisor on $X$ and let $N$ be any nef Cartier divisor. The following holds.
\begin{itemize} 
\item If $X$ is neither of general type nor quasi-elliptic with $\kappa(X)=1$, then 
\begin{itemize}
	\item[] $2mK_X+8mL+N$ is base point free, and
	\item[] $7mK_X+27mL+N$ is very ample.
\end{itemize}
\item If $p=3$ and $X$ is quasi-elliptic with $\kappa(X)=1$, then
\begin{itemize}
	\item[] $2mK_X+12mL+N$ is base point free, and
	\item[] $7mK_X+39mL+N$ is very ample.
\end{itemize}
\item If $p=2$ and $X$ is quasi-elliptic with $\kappa(X)=1$, then
\begin{itemize}
	\item[] $2mK_X+16mL+N$ is base point free, and
	\item[] $7mK_X+51mL+N$ is very ample.
\end{itemize}
\item If $p \geq 3$ and $X$ is of general type, then
\begin{itemize}
	\item[] $4mK_X+14mL+N$ is base point free, and
	\item[] $13mK_X+45mL+N$ is very ample.
\end{itemize} 
\item  If $p =2 $ and $X$ is of general type, then
\begin{itemize}
	\item[] $4mK_X+22mL+N$ is base point free, and
	\item[] $13mK_X+69mL+N$ is very ample.
\end{itemize} 
\end{itemize}
\end{theorem}
The bounds are rough. The theorem is a direct consequence of the following proposition.

\begin{proposition} \label{proposition:part_one_upgraded} Let $X$ be a normal projective surface defined over an algebraically closed field of characteristic $p>0$. Assume that $mK_X$ is Cartier for some $m\in \mbN$. Let $A$ be an ample Cartier divisor on $X$. Then 
\[
\mathrm{Bs}(m(aK_{X} + bA)+N) \subseteq \Sing(X),
\]
for any nef Cartier divisor $N$, where
\begin{itemize}
	\item $a=1, b=4$, if $X$ is neither of general type nor quasi-elliptic with $\kappa(X)=1$,
	\item $a=1, b=6$, if $X$ is quasi-elliptic with $\kappa(X)=1$ and $p=3$,
	\item $a=1, b=8$, if $X$ is quasi-elliptic with $\kappa(X)=1$ and $p=2$,
	\item $a=2, b=7$, if $X$ is of general type and $p\geq 3$,
	\item $a=2, b=11$, if $X$ is of general type and $p =2$.
\end{itemize}
\end{proposition}
\begin{proof}
This follows from Proposition \ref{proposition:andrea_general}, by exactly the same proof as of Proposition \ref{proposition:part_one}. 
\end{proof}

\begin{proof}[Proof of Theorem \ref{theorem:main}]
This follows directly from Theorem \ref{theorem:2dimp>5}, Proposition \ref{proposition:part_one_upgraded}, Proposition \ref{proposition:part_two} and Proposition \ref{proposition:very_ampleness}.

\end{proof}

\section{Matsusaka-type bounds}
The goal of this section is to prove Corollary \ref{corollary:Matsusaka-mini}. The key part of the proof is the following proposition.
\begin{proposition} Let $A$ be an ample Cartier divisor and let $N$ be a nef Cartier divisor on a normal projective surface $X$.  Then $kA-N$ is nef for any
\[
k \geq \frac{2A \cdot N}{A^2}((K_X + 3A) \cdot A +1) + 1.
\]
\end{proposition}
\begin{proof}
The proof is exactly the same as \cite[Theorem 3.3]{DiCerboFanelli}. The only difference is that for singular surfaces, the cone theorem is weaker, so we have $K_X+3D$ in the statement, instead of $K_X+2D$.
\end{proof}

The following proof is exactly the same as of \cite[Theorem 1.2]{DiCerboFanelli}.
\begin{proof}[Proof of Proposition \ref{corollary:Matsusaka-mini}]
By Theorem \ref{theorem:main-mini}, we know that $H$ is very ample. By the above proposition, we know that $kA - (H+N)$ is a nef Cartier divisor. Thus, by the proof of Theorem \ref{theorem:main-mini} 
\[
H + \underbrace{(kA - (H+N))}_{\text{nef}} = kA - N
\]
is very ample.
\end{proof}

Applying Theorem \ref{theorem:main}, we obtain the following.
\begin{corollary} \label{corollary:Matsusaka}Let $A$ and $N$ be respectively an ample and a nef Cartier divisor on an $F$-pure $\mbQ$-factorial projective surface defined over an algebraically closed field of characteristic $p>0$. Let $m\in \mbN$ be such that $mK_X$ is Cartier. Then $kA - N$ is very ample for any
\[
k > \frac{2A \cdot (H+N)}{A^2}((K_X +3A)
\cdot A + 1),
\]
where 
\begin{itemize}
\item $H \defeq 7mK_X + 27mA$, if $X$ is neither quasi-elliptic with $\kappa(X)=1$, nor of general type,
\item $H \defeq 7mK_X+39mA$, if $X$ is quasi-elliptic with $\kappa(X)=1$ and $p=3$,
\item $H \defeq 7mK_X+51mA$, if $X$ is quasi-elliptic with $\kappa(X)=1$ and $p=2$,
\item $H \defeq 13mK_X+45mA$, if $X$ is of general type and $p \geq 3$,
\item $H \defeq 13mK_X+69mA$, if $X$ is of general type and $p = 2$.
\end{itemize}
\end{corollary}

\section{Bounds on log del Pezzo pairs}
The goal of this section is to prove Corollary \ref{corollary:boundedness_of_log_del_Pezzo}. We need the following facts.
\begin{proposition} \label{proposition:log_del_pezzo_first_bounds} Let $(X,\Delta)$ be an $\epsilon$-klt log del Pezzo pair for $0< \epsilon < 3^{-1/2}$. Let $\pi \colon \ovX \to X$ be the minimal resolution. Then
\begin{itemize}
	\item[(a)]  $0 \leq (K_X + \Delta)^2 \leq \max\Big(9, \lfloor 2/\epsilon \rfloor + 4 + \frac{4}{\lfloor 2/\epsilon\rfloor}\Big)$,
	\item[(b)]  $\mathrm{rk} \Pic(\ovX) \leq 128(1/\epsilon)^5$,
	\item[(c)] $2 \leq -E^2 \leq 2 / \epsilon$ for any exceptional curve $E$ of $\pi \colon \ovX \to X$,
	\item[(d)] if $m$ is the $\mbQ$-factorial index at some point $x \in X$, then
	\[
		m \leq 2(2/\epsilon)^{128/\epsilon^5}.
	\]
\end{itemize}
\end{proposition}
\begin{proof}
Point (a) follows from \cite[Theorem 1.3]{Jiang}. Points (b) and (c) follow from \cite[Corollary 1.10]{MoriAlexeev} and \cite[Lemma 1.2]{MoriAlexeev}, respectively. Last, (d) follows from the fact that the Cartier index of a divisor divides the determinant of the intersection matrix of the minimal resolution of a singularity (see also the paragraph below \cite[Theorem A]{Lai}).
\end{proof}

Further, we need to prove the following:
\begin{lemma} \label{lemma:log_del_pezzo_second_bounds} Let $(X,\Delta)$ be a klt log del Pezzo pair such that $m(K_X+\Delta)$ is Cartier for some natural number $m \geq 2$. Then
\begin{enumerate}
	\item  $0 \leq (K_X + \Delta) \cdot K_X \leq  3m\max\big(9, 2m  + 4 + \frac{2}{m}\big)$, and 
	\item $|K_X^2| \leq 128m^5(2m-1)$.
\end{enumerate}
\end{lemma}
Recall that if a log del Pezzo pair $(X,\Delta)$, with Cartier index $m$, is klt, then it must be $1/m$-klt.
\begin{proof}
 The non-negativity in  $(1)$ is clear, since 
\[
\big(K_X + \Delta\big) \cdot K_X = \big(K_X+\Delta\big)^2 - \big(K_X+\Delta\big)\cdot \Delta \geq 0.
\]
Further, by cone theorem, $K_{X} - 3m\big(K_{X} + \Delta\big)$ is nef, and thus
\[
(K_{X} + \Delta) \cdot \big(K_X - 3m(K_{X} + \Delta))  \leq 0.
\]
This, together with (a) in Proposition \ref{proposition:log_del_pezzo_first_bounds}, implies $(1)$.

To prove $(2)$, we proceed as follows. Let $\pi \colon \ovX \to X$ be the minimal resolution of singularities of $X$. By (b) in Proposition \ref{proposition:log_del_pezzo_first_bounds}, we have $\mathrm{rk} \Pic(\ovX) \leq 128m^5$, and so  $-9 \leq -K_{\ovX}^2  \leq 128m^5$. Indeed, the self intesection of the canonical bundle on a minimal model of a rational surface is $8$ or $9$, and each blow-up decreases it by one.

Write 
\[
K_{\ovX} + \sum a_i E_i = \pi^*K_X,
\]
where $E_i$ are the exceptional divisors of $\pi$. Notice, that since $\ovX \to X$ is minimal and $X$ is klt, we have $0 \leq a_i < 1$. By applying (b) and (c) from Proposition \ref{proposition:log_del_pezzo_first_bounds}, we obtain
\begin{align*}
|K_X^2| &= \big|\big(K_{\ovX} + \sum a_i E_i\big) \cdot K_{\ovX} \big|\\ &\leq \big|K_{\ovX}\big|^2 + 128m^5(2m-2)\\ &\leq 128m^5(2m-1).
\end{align*}
\end{proof}

\begin{proof}[Proof of Corollary \ref{corollary:boundedness_of_log_del_Pezzo}]
By Proposition \ref{proposition:log_del_pezzo_first_bounds}, the $\mathbb{Q}$-factorial index of $X$ is bounded with respect to $\epsilon$. Indeed, the $\mathbb{Q}$-factorial index is bounded at each point by (d), and the number of singular points is bounded by (b). Hence, we can assume that there exists $m\in \mbN$ bounded with respect to $\epsilon$ and $I$, such that $mK_X$, $m\Delta$, and $m(K_X+\Delta)$ are Cartier. 

Set $a = 13m$ and $b=45m^2$. By Theorem \ref{theorem:main-mini} and  Remark \ref{remark:effective_vanishing}, the divisor $H \defeq aK_X - b(K_X + \Delta)$ is very ample, and $H^i(X, H)=0$ for $i>0$. Proposition \ref{proposition:log_del_pezzo_first_bounds}(a) and Lemma \ref{lemma:log_del_pezzo_second_bounds} imply that $H^2$, $|H \cdot K_X|$, $H \cdot \Delta$, $|K_X \cdot \Delta|$, and $|\Delta^2|$  are bounded with respect to $m$.

The ample divisor $H$ embeds $X$ into a projective space of dimension $\chi(\mcO_X(H)) = H^0(X,\mcO_X(H))$, which is bounded with respect to $m$ by the Riemann-Roch theorem.

\end{proof}
If $mK_X$ and $m(K_X+\Delta)$ are Cartier for $m>1$, then one can easily calculate, that is enough to take
\[
b(\epsilon, I) = \Big(a^2+b^2\Big)\Big(128m^5(2m-1) + \max\Big(9,2m + 4 + \frac{2}{m}\Big)\Big),
\]
where $a=13m$ and $b=45m^2$.

\begin{remark}Corollary \ref{corollary:boundedness_of_log_del_Pezzo} and the Riemann-Roch theorem for surfaces with rational singularities imply that the absolute values of the coefficients of the Hilbert polynomial of $X$ with respect to $H$ are bounded with respect to $\epsilon$ and $I$. Further, let $n \in \mbN$ be such that $n\Delta$ is Cartier. Then the absolute values of the coefficients of the Hilbert polynomial of $n\Delta$ with respect to $H|_{n\Delta}$ are bounded with respect to $\epsilon$, $I$, and $n$. Indeed,
\[
\chi(\mcO_{n\Delta}(mH)) = mn\Delta \cdot H - \frac{1}{2}n\Delta \cdot (n\Delta + K_X)
\]
for $m \in \mbZ$, by the Riemann-Roch theorem and the adjunction formula.

\end{remark}

\begin{remark} One of the reasons for our interest in the above corollary is that it provides bounds on $\epsilon$-klt log del Pezzo surfaces which are independent of the characteristic. In particular, it shows the existence of a bounded family of $\epsilon$-klt log del Pezzo surfaces over $\Spec \mbZ$ (see \cite[Lemma 3.1]{CTW15b}). We were not able to verify whether Koll\'{a}r's bounds depend on the characteristic or not. We believe that stating explicit bounds might ease the life of future researchers, wanting to handle questions related to the behavior of log del Pezzo surfaces in mix characteristic or for big enough characteristic.
\end{remark}

\subsection{Effective vanishing of $H^1$}
The goal of this subsection is to give a proof of the following proposition.

\begin{proposition} \label{proposition:effective_vanishing} Let $X$ be a normal projective surface. Then
\[
H^i(X, \mathcal{O}_X(D)) = 0 \text{ for } i>0,
\]
where $D = 3K_X + 14A + N$ is Cartier for an ample Cartier divisor $A$ and a nef $\mbQ$-Cartier divisor $N$.
\end{proposition}

It was pointed to us by the anonymous referee that the proposition follows from a result of Koll{\'a}r. We present this approach below.

First, let us recall the aforementioned result. 
\begin{theorem}[{\cite[Theorem II.6.2 and Remark II.6.7.2]{kollarcurves}}] \label{theorem:kollar_surfaces}
Let $X$ be a normal, projective variety defined over a field of characteristic $p$. Let $L$ be an ample $\mbQ$-Cartier  Weil divisor on $X$ satisfying $H^1(X, \mcO_X(-L)) \neq 0$. Assume that $X$ is covered by a family of curves $\{D_t\}$ such that $X$ is smooth along a general $D_t$ and 
\[
((p-1)L - K_X) \cdot D_t > 0.
\]
Then, through every point $x \in X$ there is a rational curve $C \subseteq X$ such that 
\[
L \cdot C \leq 2\dim X \frac{L \cdot D_t}{((p-1)L - K_X)\cdot D_t}. 
\]
\end{theorem}

\begin{proof}[{Proof of Proposition \ref{proposition:effective_vanishing}}]
Set $L = 2K_X + 14A+N$. Since $D$ is Cartier and $\omega_X$ is reflexive, we get that $\omega_X \otimes \mcO_X(-D) = \mcO_X(-L)$. Thus, by Serre duality, we need to show $H^i(X,\mcO_X(-L)) = 0$ for $i<2$. However, by cone theorem, we have that $K_X+3A$ is nef, and so $L$ is ample, giving in particular that $H^0(X, \mcO_X(-L)) = 0$. Hence, we are left to show $H^1(X,\mcO_X(-L))=0$.

To this end, we suppose by contradiction that $H^1(X,\mcO_X(-L))\neq 0$ and apply Theorem \ref{theorem:kollar_surfaces} for a general pencil of curves $\{D_t\}$ in some very ample linear system. Since $(p-1)L-K_X$ is ample, the assumptions of the theorem are satisfied. In particular, we get a curve $C$ such that
\[
L \cdot C \leq 4 \frac{L \cdot D_t}{((p-1)L-K_X) \cdot D_t},
\]
and, as $L$ is ample, this in turn gives
\[
L \cdot C  \leq 4 \frac{L \cdot D_t}{(L - K_X) \cdot D_t}.
\]

Since $K_X + 3A$ is nef, we have $L\cdot C \geq 8$. Further, $K_X \cdot D_t < (L-K_X)\cdot D_t$, as $L-K_X = K_X + 14A+N$. Therefore,
\[
8 \leq L \cdot C < 4\Bigg(\frac{2(L-K_X)\cdot D_t}{(L-K_X)\cdot D_t}\Bigg) = 8, 
\]
which is a contradiction.
\end{proof}
\begin{remark} \label{remark:effective_vanishing} Proposition \ref{proposition:effective_vanishing} shows, under assumptions and notation of Theorem \ref{theorem:main-mini}, that the very ample divisor $H \coloneq 13mK_X + 45mA$ satisfies $H^i(X,H) = 0$ for $i>0$. Similar statements hold for very ample divisors considered in Theorem \ref{theorem:main}. 
\end{remark}    

\section*{Acknowledgements}
I would like to express my enormous gratitude to Paolo Cascini for his indispensable help. His guidance and suggestions played a crucial role in my research.

Further, I would like to thank Hiromu Tanaka for our ample discussions, from which I benefited immensely. I also thank Roland Abuaf, Andrea Fanelli, Ching-Jui Lai and the anonymous referee for comments and helpful suggestions.

\bibliographystyle{amsalpha}
\bibliography{EffectiveBoundsJakubWitaszek}

\end{document}